\documentclass[a4paper,12pt]{amsart}
\usepackage{amssymb}
\usepackage{cite}
\usepackage{ifthen}
\usepackage[dvips]{graphicx}
\usepackage{tabularx}

\nonstopmode \numberwithin{equation}{section}
\usepackage{amssymb}
\usepackage{ifthen}
\usepackage{graphicx}
\usepackage{amsmath}
\usepackage[T1]{fontenc} 
\usepackage[utf8]{inputenc}
\usepackage[usenames,dvipsnames]{color}
\usepackage{color}
\usepackage[english]{babel}
\usepackage{fancyhdr}
\usepackage{fancybox}
\usepackage{tikz}

\theoremstyle{plain}
\newtheorem{prop}{Proposition}

\newtheorem{conj}{Conjecture}

\theoremstyle{definition}

\newtheorem{cor}{Corollary}[section]
\newtheorem{thm}{Theorem}[section]

\newtheorem{prob}{Problem}[section]

\newtheorem{lem}{Lemma}[section]
\newtheorem{rem}{Remark}[section]

\theoremstyle{plain}
\newtheorem*{thmA}{Theorem A}

\newtheorem*{lemA}{Lemma A}
\newtheorem*{lemB}{Lemma B}

\newcounter{minutes}
\divide\time by 60
\newcounter{hours}
\multiply\time by 60
\addtocounter{minutes}{-\time}

\newcounter {own}
\def\theown {\thesection       .\arabic{own}}

\newenvironment{pf}[1][]{%
	\vskip 3mm
	\noindent
	\ifthenelse{\equal{#1}{}}%
	{{\slshape Proof. }}%
	{{\slshape #1.} }%
}%
{\qed\bigskip}

\newcounter{alphabet}

\def\be{\begin{equation}}
	\def\ee{\end{equation}}

\newcommand{\bee}{\begin{enumerate}}
	\newcommand{\eee}{\end{enumerate}}

\newcommand{\blem}{\begin{lem}}
	\newcommand{\elem}{\end{lem}}
\newcommand{\bthm}{\begin{thm}}
	\newcommand{\ethm}{\end{thm}}
\newcommand{\bcor}{\begin{cor}}
	\newcommand{\ecor}{\end{cor}}
\newcommand{\beg}{\begin{examp}}
	\newcommand{\eeg}{\end{examp}}
\newcommand{\begs}{\begin{examples}}
	\newcommand{\eegs}{\end{examples}}

\newcommand{\bdefn}{\begin{defn}}
	\newcommand{\edefn}{\end{defn}}

\newcommand{\bprob}{\begin{prob}}
	\newcommand{\eprob}{\end{prob}}
\newcommand{\bei}{\begin{itemize}}
	\newcommand{\eei}{\end{itemize}}

\newcommand{\bcon}{\begin{conj}}
	\newcommand{\econ}{\end{conj}}
\newcommand{\bcons}{\begin{conjs}}
	\newcommand{\econs}{\end{conjs}}
\newcommand{\bprop}{\begin{prop}}
	\newcommand{\eprop}{\end{prop}}
\newcommand{\br}{\begin{rem}}
	\newcommand{\er}{\end{rem}}
\newcommand{\brs}{\begin{rems}}
	\newcommand{\ers}{\end{rems}}
\newcommand{\bo}{\begin{obser}}
	\newcommand{\eo}{\end{obser}}
\newcommand{\bos}{\begin{obsers}}
	\newcommand{\eos}{\end{obsers}}
\newcommand{\bpf}{\begin{pf}}
	\newcommand{\epf}{\end{pf}}
\newcommand{\ba}{\begin{array}}
	\newcommand{\ea}{\end{array}}
\newcommand{\beq}{\begin{eqnarray}}
	\newcommand{\beqq}{\begin{eqnarray*}}
		\newcommand{\eeq}{\end{eqnarray}}
	\newcommand{\eeqq}{\end{eqnarray*}}

\begin{document}

\title{Sharp Bohr and Bohr-Rogosinski inequalities involving area measure for close-to-convex harmonic mappings}

\author{Molla Basir Ahamed$^*$}
\address{Molla Basir Ahamed, Department of Mathematics, Jadavpur University, Kolkata-700032, West Bengal, India.}
\email{mbahamed.math@jadavpuruniversity.in}

\author{Partha Pratim Roy}
\address{Partha Pratim Roy, Department of Mathematics, Jadavpur University, Kolkata-700032, West Bengal, India.}
\email{parthacob2023@gmail.com}

\subjclass[{AMS} Subject Classification:]{Primary 30C45, 30C50, 30C80, Secondary 31A05, 30C35}

\def\thefootnote{}
\footnotetext{ {\tiny File:~\jobname.tex,
printed: \number\year-\number\month-\number\day,
          \thehours.\ifnum\theminutes<10{0}\fi\theminutes }
}\makeatletter\def\thefootnote{\@arabic\c@footnote}\makeatother
\begin{abstract}
	In this article, we investigate refined and generalized versions of the Bohr and Bohr--Rogosinski inequalities for a normalized subclass $\mathcal{P}_{\mathcal{H}}^{0}(\alpha)$ ($0 \le \alpha < 1$) of univalent close-to-convex harmonic mappings $f = h + \overline{g}$ defined on the open unit disk $\mathbb{D} \subset \mathbb{C}$. By incorporating non-negative monotone increasing functions associated with the planar area integral $S_r/\pi$ of the image domain $f(\mathbb{D}_r)$, we establish new sharp Bohr-type inequalities expressed in terms of the Euclidean distance $d(f(0), \partial f(\mathbb{D}))$. Furthermore, we establish sharp Bohr--Rogosinski-type inequalities involving powers of the modulus of the mapping $|f(z)|^p$ ($p \ge 1$). All associated radii are proven to be sharp, and extremal functions realizing the equality cases are explicitly identified. As applications, our results generalize and unify several well-known classical and recent theorems in geometric function theory.
\end{abstract}
\maketitle
\pagestyle{myheadings}
\markboth{M. B. Ahamed and P. P. Roy}{Sharp Bohr and Bohr-Rogosinski inequalities involving area measure }
\tableofcontents
\section{\bf Introduction}
Let $\mathbb{D} := \{z \in \mathbb{C} : |z| < 1\}$ be the open unit disk in the complex plane $\mathbb{C}$. Let $\mathcal{B}(\mathbb{D})$ denote the class of analytic functions $f$ in the unit disk $\mathbb{D}$ such that $|f(z)| \le 1$ in $\mathbb{D}$. One of the most celebrated classical results for the class $\mathcal{B}(\mathbb{D})$ was given by H. Bohr.

\noindent\textbf{Theorem 1.1.} \cite{Bohr-PLMS-1914} \textit{If $f(z) = \sum_{n=0}^{\infty} a_n z^n \in \mathcal{B}(\mathbb{D})$, then}
\begin{align}\label{eq-1.1}
	M_f(r) := \sum_{n=0}^{\infty} |a_n| r^n \le 1 \quad \text{for} \quad |z| = r \le \frac{1}{3}.
\end{align}
\textit{The constant $1/3$ is best possible.}

\medskip
Interest in the Bohr inequality was revived when Dixon (see \cite{Dixon-BLMS-1995}) used it to settle in the negative the conjecture that a non-unital Banach algebra that satisfies the von Neumann inequality must be isometrically isomorphic to a closed subalgebra of $B(\mathcal{H})$ for some Hilbert space $\mathcal{H}$. Subsequently, Paulsen and Singh \cite{Paulsen-Singh-PAMS-2004} have extended the Bohr inequality in the context of Banach algebra. An open problem on the Bohr inequality has been raised in \cite{Djakov-Ramanujan-JA-2000} has been settled in \cite{Kayumov-Ponnusamy-AASF-2019}. Bohr radius has also been investigated by many researchers in various multidimensional spaces (see e.g., \cite{Ahammed-Ahamed-CVEE-2024, Aizenberg-PAMS-2000, Boas-Khavinson-PAMS-1997, Galicer-Mansilla-Muro-TAMS-2020, Kumar-PAMS-2023, Liu-Ponnusamy-PAMS-2021}) and references therein.

Similar to Bohr inequalities, there is another concept known as the Bohr--Rogosinski inequalities, which were established by Kayumov \textit{et al.} \cite{Kayumov-Khammatova-Ponnusamy-JMAA-2021} for the class $\mathcal{B}(\mathbb{D})$. For analytic functions $f$ defined on the unit disk $\mathbb{D}$, $S_r := S_r(f)$ denotes the planar integral
\[
S_r = \int_{\mathbb{D}_r} |f'(z)|^2 dA(z), \quad \text{where } \mathbb{D}_r := \mathbb{D}(0;r) \text{ and } 0 < r < 1.
\]
If $f(z) = \sum_{n=0}^{\infty} a_n z^n$, then the quantity $S_r$ has the series representation $S_r = \pi \sum_{n=1}^{\infty} n|a_n|^2 r^{2n}$, and the quantity $S_r$ plays a significant role in the study of improved Bohr inequalities. For example, Kayumov and Ponnusamy (see \cite{Kayumov-Ponnusamy-CRAS-2018}) obtained the result improving the classical Bohr inequality for the class $\mathcal{B}(\mathbb{D})$. Later, Ismagilov \textit{et al.} (see \cite{Ismagilov-Kayumov-Ponnusamy-JMAA-2020}) established another improved version of the classical Bohr inequality, and also derived a sharp inequality as an improvement of the Bohr--Rogosinski inequality for the class $\mathcal{B}(\mathbb{D})$.

\subsection{\bf Close-to-convex harmonic mappings} For a continuously differentiable complex-valued mapping 
$f(z) = u(z) + iv(z)$, $z = x + iy$, 
we use the common notions for its formal derivatives:

\begin{align*}
	f_z = \frac{1}{2}(f_x - i f_y) 
	\quad \text{and} \quad 
	f_{\bar{z}} = \frac{1}{2}(f_x + i f_y).
\end{align*} 
In complex analysis, a complex-valued function $f = u + iv$ defined on a domain $D \subset \mathbb{C}$ (where $u$ and $v$ are real-valued functions) is called a harmonic mapping if both its real part $u$ and imaginary part $v$ are real harmonic functions in $D$.This means they satisfy the classical Laplace equation
\begin{align*}
	\Delta u = \frac{\partial^2 u}{\partial x^2} + \frac{\partial^2 u}{\partial y^2} = 0 \quad \text{and} \quad \Delta v = \frac{\partial^2 v}{\partial x^2} + \frac{\partial^2 v}{\partial y^2} = 0.
\end{align*}
Therefore, the Laplace equation $\Delta f = 0$ for a complex-valued function $f$ takes the elegant and simple form 
\begin{align*}
	\frac{\partial^2 f}{\partial z \partial \bar{z}} = 0 \quad \text{or} \quad f_{z\bar{z}} = 0.
\end{align*}

While harmonic mappings are deep, elegant structures in geometric function theory, they are not just purely abstract tools. Because they satisfy the Laplace equation while allowing a non-analytic (co-analytic) component, they provide a much more flexible framework than classical conformal (analytic) mappings for solving physical and physical-geometric problems. In fact, harmonic mappings play a crucial role due to their wide range of applications across various branches of science and engineering. In particular, techniques involving harmonic mappings have been effectively employed in the study and analysis of fluid flow problems (see \cite{Aleman-2012,Constantin-2017}). For instance, Aleman and Constantin~\cite{Aleman-2012} established a fundamental link between harmonic mappings and ideal fluid flows. Moreover, they introduced an elegant method to express and solve the incompressible two-dimensional Euler equations in terms of univalent harmonic mappings (refer to \cite{Constantin-2017} for further details).
\vspace{1.2mm}
Let $\mathcal{H}(\Omega)$ denote the class of complex-valued harmonic functions defined on a domain $\Omega$.  
It is well-known that every function $f \in \mathcal{H}(\Omega)$ admits the canonical decomposition  $f = h + \overline{g},$
where both $h$ and $g$ are analytic in $\Omega$.  
According to the classical result of Lewy~\cite{Lew-BAMS-1936}, a harmonic mapping $f = h + \overline{g}$ is locally univalent in $\Omega$ if, and only if, the Jacobian determinant  
\begin{align}\label{Eqn-1.1}
	J_f(z) := |f_z(z)|^2 - |f_{\bar{z}}(z)|^2 = |h'(z)|^2 - |g'(z)|^2
\end{align}
is non-vanishing throughout $\Omega$.  
Furthermore, a locally univalent harmonic mapping is said to be \textit{sense-preserving} if $J_f(z)>0$ .\vspace{2mm}

A complex-valued harmonic mapping $f$ defined on $\mathbb{D}$ is called a close-to-convex if 
\begin{enumerate}
	\item[(i)] $f(z)$ is univalent (one-to-one) in $\mathbb{D}$,
	\item[(II)] $f(z)$ is sense-preserving, and
	\item[(iii)] The image domain $f(\mathbb{D})$ is a close-to-convex domain.
\end{enumerate}
Geometrically, a domain $\Omega \subset \mathbb{C}$ is \textit{close-to-convex} if its complement $\mathbb{C} \setminus \Omega$ can be partitioned into a family of non-intersecting, straight half-lines (rays) that stretch out to infinity.
\begin{table}[htbp]
	\centering
	\caption{Comparison between analytic and harmonic close-to-convex functions}
	\label{tab:close_to_convex}
	\begin{tabularx}{\textwidth}{|l|X|X|}
		\hline
		\textbf{Property} & \textbf{Analytic functions ($f$)} & \textbf{Harmonic mappings ($f = h + \overline{g}$)} \\ \hline
		\textbf{Structure} & Purely analytic. & Split into analytic ($h$) and co-analytic ($g$) parts. \\ \hline
		\textbf{Condition} & $\operatorname{Re}\left(\frac{f'(z)}{\phi'(z)}\right) > 0$ (for convex $\phi$) & $\operatorname{Re}\left(\frac{h'(z) - \epsilon g'(z)}{\phi'(z)}\right) > 0$ (for $|\epsilon|=1$) \\ \hline
	\end{tabularx}
\end{table}
\subsection{\bf Harmonic mappings and Bohr inequality on $\mathbb{D}$}
A function $f \in \mathcal{C}^2$ defined on a domain $\Omega \subseteq \mathbb{C}$ is termed harmonic whenever it satisfies the Laplace equation $\Delta f = 4f_{z\overline{z}} = 0$. Let $\operatorname{Har}(\Omega)$ represent the family of orientation-preserving harmonic mappings in a simply connected domain $\Omega$, and let $\mathcal{H}$ denote the subclass consisting of such mappings. It is an established property that every $f \in \operatorname{Har}(\mathbb{D})$ admits a canonical representation $f = h + \overline{g}$, where $h$ and $g$ are analytic functions in the unit disk $\mathbb{D}$ normalized by $h(0) = 0 = h'(0) - 1$ and $g(0) = 0$. In particular, we define the subclass $\mathcal{H}_0 \subset \mathcal{H}$ by prescribing the additional normalization $g'(0) = 0$. Consequently, every mapping $f = h + \overline{g} \in \mathcal{H}_0$ possesses the series expansion
\begin{align}\label{eq-1.3}
	h(z) = z + \sum_{n=2}^{\infty} a_n z^n \quad \text{and} \quad g(z) = \sum_{n=1}^{\infty} b_n z^n.
\end{align}
Furthermore, the classical Bohr inequality can be recast geometrically in the equivalent formulation
\begin{align}\label{eq-1.4}
	\sum_{n=1}^{\infty} |a_n| r^n \le 1 - |f(0)| = d\big(f(0), \partial\mathbb{D}\big) \quad \text{for } |z| = r \le \frac{1}{3},
\end{align}
where $d\big(f(0), \partial\mathbb{D}\big)$ signifies the standard Euclidean distance from $f(0)$ to the boundary of the unit disk $\mathbb{D}$.\vspace{1.2mm}

The Bohr inequality for a class of harmonic mappings was initiated first in \cite{Abu-Muhanna-2010} in terms of the Euclidean distance $d$ and was investigated in \cite{Kayumov-Ponnusamy-Shakirov-MN-2018} and subsequently by a number of authors (see e.g., \cite{Ahamed-CMFT-2022, Ahamed-Ahammed-MJM-2024, Ahamed-Allu-BMMSS-2022, Ahamed-Allu-Halder-AMP-2021, Ahamed-Allu-Halder-CVEE-2023, Alkhaleefah-Kayumov-Ponnusamy-PAMS-2019, Allu-Halder-BSM-2021, Evdoridis-Ponnusamy-Rasila-IM-2019, Huang-Liu-Ponnusamy-MJM-2021, Kayumov-Ponnusamy-JMAA-2018, Liu-Ponnusamy-BMMSS-2019}).\vspace{1.2mm}

 \subsection{\bf  Bohr-Rogosinski inequality for the class $\mathcal{P}_{\mathcal{H}}^{0}(\alpha)$ } In this paper, our primary objective is to study improved versions of the classical Bohr inequality for certain class of harmonic mappings. In $2013$, Ponnusamy \textit{et al.}~\cite{Ponnusamy-CVEE-2013} introduced the following subclasses:
\begin{align*}
	\mathcal{P}_{\mathcal{H}} = 
	\{ f = h + \overline{g} \in \mathcal{H} : 
	\operatorname{Re} h'(z) > |g'(z)|, \ z \in \mathbb{D} \},
\end{align*}
and $\mathcal{P}_{\mathcal{H}}^{0} = \mathcal{P}_{\mathcal{H}} \cap \mathcal{H}_0$.  
Motivated by these classes, Li and Ponnusamy~\cite{Li-Ponnusamy-2016} further investigated related subclasses of harmonic mappings.
Let  
\begin{align*}
	\mathcal{P}_{\mathcal{H}}(\alpha) = 
	\{ f = h + \overline{g} \in \mathcal{H} : 
	\operatorname{Re}(h'(z) - \alpha) > |g'(z)|, \ z \in \mathbb{D} \},
\end{align*}
and define $\mathcal{P}_{\mathcal{H}}^{0}(\alpha) = \mathcal{P}_{\mathcal{H}}(\alpha) \cap \mathcal{H}_0$, where 
\begin{align}\label{Eqn-3.2}
	\mathcal{P}_{\mathcal{H}}^{0}(\alpha) = 
	\left\{ f = h + \overline{g} \in \mathcal{H} : 
	\operatorname{Re}\,(h'(z) - \alpha) > |g'(z)|, \ 
	0 \le \alpha < 1, \ g'(0) = 0 \ \text{for} \ z \in \mathbb{D} \right\}.
\end{align}
It is clear that $\mathcal{P}_{\mathcal{H}}(\alpha) \subseteq \mathcal{P}_{\mathcal{H}}$ and  
$\mathcal{P}_{\mathcal{H}}^{0}(\alpha) \subseteq \mathcal{P}_{\mathcal{H}}^{0}$ for $0 \leq \alpha < 1$. Li and Ponnusamy~\cite{Li-Ponnusamy-NA-2013} proved that functions in  
$\mathcal{P}_{\mathcal{H}}^{0}(\alpha)$ are univalent for $0 \leq \alpha < 1$. They also established coefficient estimates and investigated the univalence of the partial sums for this class.\vspace{1.2mm}

Motivated by the class $\mathcal{P}_{\mathcal{H}}^{0}$ \cite{Ponnusamy-CVEE-2013}, Li and Ponnusamy \cite{Li-Ponnusamy-NA-2013} studied the growth estimates and sharp coefficient bounds of functions $f$ in the class $\mathcal{P}_{\mathcal{H}}^{0}(\alpha)$, which is defined by \eqref{Eqn-3.2}. The Bohr phenomenon for the class $\mathcal{P}_{\mathcal{H}}^{0}(\alpha)$ was recently studied in the paper \cite{Ahamed-Allu-Halder-AMP-2021}. The key ingredients of our investigation into improved Bohr inequalities for functions in $\mathcal{P}_{\mathcal{H}}^{0}(\alpha)$ are the following coefficient bounds and growth estimates, which were proved by Li and Ponnusamy \cite{Li-Ponnusamy-NA-2013} and Allu and Halder \cite{Allu-Halder-BSM-2021}, respectively. 

\begin{lemA} (see \cite[Theorem 2]{Li-Ponnusamy-NA-2013})
	Let $f=h+\overline{g} \in \mathcal{P}^{0}_{\mathcal{H}}(\alpha)$ with $0\leq \alpha <1$. Then for any $k \geq 2$, 
	\begin{enumerate}
		\item[(i)] $\displaystyle |a_k| + |b_k|\leq \frac {2(1-\alpha)}{k}; $\\[1mm]
		
		\item[(ii)] $\displaystyle ||a_k| - |b_k||\leq \frac {2(1-\alpha)}{k};$\\[1mm]
		
		\item[(iii)] $\displaystyle |a_k|\leq \frac {2(1-\alpha)}{k}.$
	\end{enumerate}
	All the inequalities (i), (ii) and (iii) are sharp, with 
	\begin{align}\label{eq-1.6}
		f_{\alpha}(z)=(1-\alpha)(-z-2\log(1-z))+\alpha z
	\end{align}
	being the extremal function.
\end{lemA}
The growth theorem for functions in the class $ \mathcal{P}^{0}_{\mathcal{H}}(\alpha) $, established in \cite{Allu-Halder-BSM-2021}, is essential in our study for finding the Euclidean distance $d(f(0), \partial f(\mathbb{D}))$ between $f(0)$ and the boundary of $f(\mathbb{D})$.
\begin{lemB} (see \cite[Theorem 2.1]{Allu-Halder-BSM-2021})
	Let $f=h+\overline{g} \in \mathcal{P}^{0}_{\mathcal{H}}(\alpha)$ with $0\leq \alpha <1$. Then 
	\begin{equation*}
		|z|+ \sum\limits_{k=2}^{\infty}  \dfrac{2(1-\alpha)(-1)^{k-1}}{k} |z|^{k} \leq |f(z)| \leq |z|+ \sum\limits_{k=2}^{\infty}  \dfrac{2(1-\alpha)}{k} |z|^{k}.
	\end{equation*} 
	Both  inequalities are sharp.
\end{lemB}
In view of the Lemma A  in~\cite{Li-Ponnusamy-NA-2013}, the Bohr inequality for the class $\mathcal{P}^{0}_{\mathcal{H}}(\alpha)$ is obtained in terms of distance formulation~\eqref{eq-1.4} and shown that the Bohr radius for the class $\mathcal{P}^{0}_{\mathcal{H}}(\alpha)$ is best possible.

\begin{thmA}[\cite{Allu-Halder-BSM-2021}, Theorem~2.2]
	Let $f \in \mathcal{P}_{\mathcal{H}}^{0}(\alpha)$ be given by \eqref{eq-1.3} with $0 \leq \alpha < 1$. Then 
	\begin{equation}\label{eq-1.7}
		|z| + \sum_{n=2}^{\infty} (|a_n| + |b_n|) |z|^n \leq d(f(0), \partial f(\mathbb{D}))
	\end{equation}
	\textit{holds for $|z| = r \le r_f$, where $r_f$ is the unique positive root of}
\begin{align*}
	r + 2(1 - \alpha)\sum_{n=2}^{\infty} \frac{r^n}{n} = 1 + 2(1 - \alpha)\sum_{n=2}^{\infty} \frac{(-1)^{n-1}}{n} 
\end{align*}
\noindent
\textit{in $(0,1)$. The radius $r_f$ is the best possible.}
\end{thmA}

\section{\bf Main Results}
In this section, we establish Bohr-type inequalities for certain expression of the quantity $S_r/\pi$ for certain class of harmonic mappings. Inspired by the results in \cite{Allu-Halder-BSM-2021} and continuing the study, the discussions above motivate us to establish improved Bohr inequalities for the class $\mathcal{P}_{\mathcal{H}}^{0}(\alpha)$.\vspace{1.2mm}

Using the Lemmas A and B, we prove the following sharp improved Bohr inequality for the class $\mathcal{P}_{\mathcal{H}}^{0}(\alpha)$.

\begin{thm}\label{Th-2.1}
	Let $f \in \mathcal{P}_{\mathcal{H}}^{0}(\alpha)$ be given by \eqref{eq-1.3} with $0 \leq \alpha < 1$ and $\mu, \lambda : [0, \infty) \to [0, \infty)$ be monotone increasing functions. Then, for $p \ge 1$, we have
	\begin{align}\label{eq-2.1}
		\mathcal{A}_{f, p, \mu, \lambda}(r) := |z|^p + \sum_{n=2}^{\infty} (|a_n| + |b_n|) |z|^n + \mu\left(\frac{S_{|z|}}{\pi}\right) + \lambda\left(\left(\frac{S_{|z|}}{\pi}\right)^2\right) \le d(f(0), \partial f(\mathbb{D})) 
	\end{align}
		holds for $|z|=r\leq R_f(\alpha)$, where $R_f(\alpha)$ is the unique root in $(0,1)$ of the equation
	\[
	J_1(r):=
	r^p - 2(1-\alpha)(r+\ln(1-r))
	+\mu\big(M_\alpha(r)\big)
	+\lambda\big(M_\alpha(r)^2\big)
	-1 -2(1-\alpha)(\ln 2 - 1)
	=0.
	\]
	The radius $R_f(\alpha)$ is best possible.
\end{thm}
\begin{table}[htbp]
	\centering
	\caption{Numerical values of the sharp radius $R_f(\alpha)$ for
		$p=1$ and $p=2$, with $\mu(s)=\frac{16}{9}s$ and $\lambda(s)=0$.}
	\label{tab:p1-p2-mu16}
	\begin{tabular}{c|cc}
		\hline
		$\alpha$ & $R_f(\alpha)$ for $p=1$ & $R_f(\alpha)$ for $p=2$ \\
		\hline
		$0.0$ & $0.217485$ & $0.281713$ \\
		$0.2$ & $0.270101$ & $0.336115$ \\
		$0.4$ & $0.324476$ & $0.393286$ \\
		$0.6$ & $0.383201$ & $0.455754$ \\
		$0.9$ & $0.483436$ & $0.562311$ \\
		\hline
	\end{tabular}
\end{table}
\begin{figure}[htbp]
	\centering
	\includegraphics[width=0.98\textwidth]{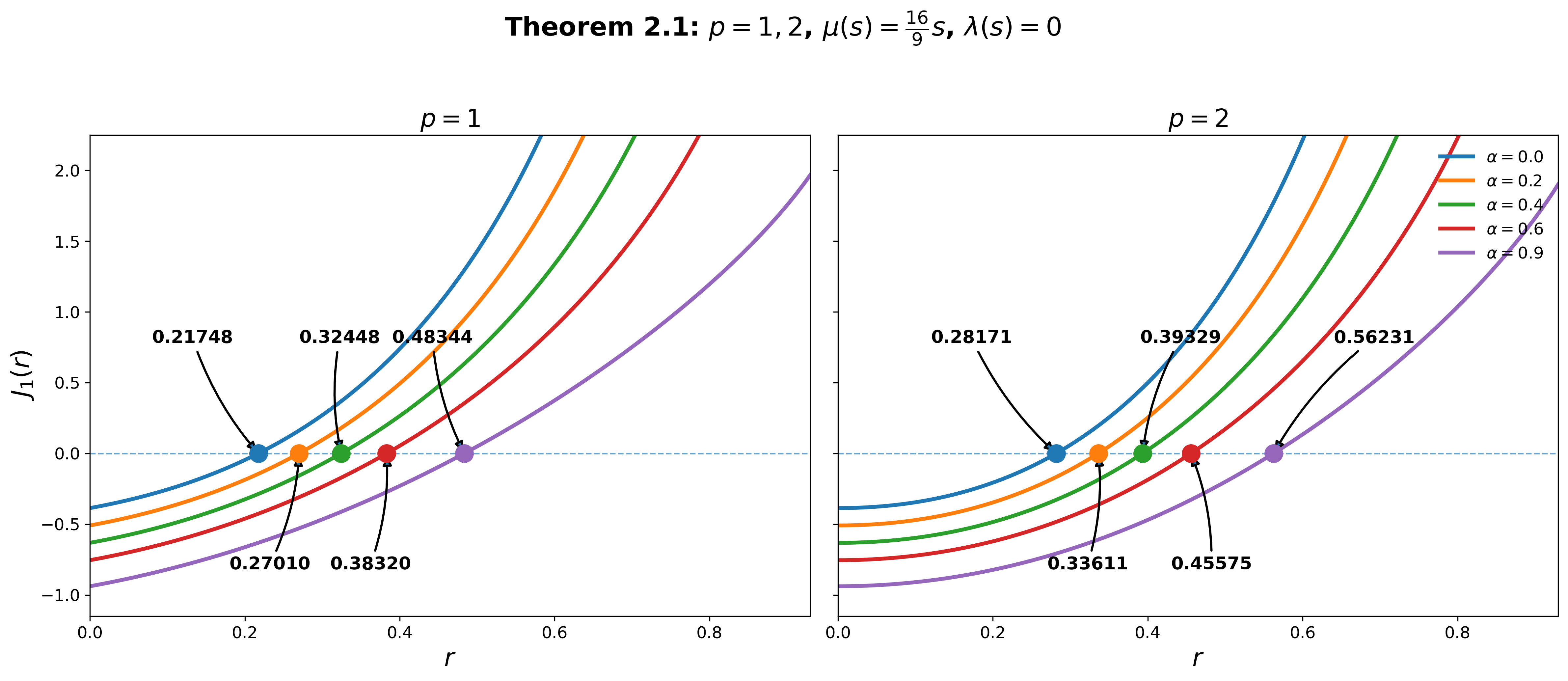}
	\caption{Graphical representation of the equation $J_1(r)=0$ in Theorem~2.1
		for $p=1$ and $p=2$, with $\mu(s)=\frac{16}{9}s$ and $\lambda(s)=0$,
		for $\alpha=0,\,0.2,\,0.4,\,0.6,$ and $0.9$. The marked points indicate
		the corresponding sharp radii $R_f(\alpha)$.}
	\label{fig:p1-p2-mu16}
\end{figure}
\begin{rem}
Theorem~\ref{Th-2.1} generalizes and refines Theorem~A in several directions:
\begin{itemize}
	\item Setting $\mu(s) = \lambda(s) \equiv 0$ for $s \in [0, \infty)$ along with $p = 1$ reduces inequality~\eqref{eq-2.1} precisely to~\eqref{eq-1.7}.
	\item For $\mu(s) = (16/9)s$, $\lambda(s) \equiv 0$, and $p = 1$, inequality~\eqref{eq-2.1} serves as the harmonic analogue of the classical result due to Kayumov and Ponnusamy \cite[Eq.~(2), Theorem~1]{Kayumov-Ponnusamy-CRAS-2018}.
	\item Choosing $\mu(s) = (16/9)s$ and $\lambda(s) = (18.6095\dots)s$ with $p = 1$ yields the harmonic counterpart of the inequality established in \cite[Eq.~(3), Theorem~1]{Ismagilov-Kayumov-Ponnusamy-JMAA-2020}.
	\item When $\lambda(s) \equiv 0$, $p = 1$, and $\mu(s) = P_k(s) = \sum_{j=1}^k \lambda_j s^j$ is a polynomial of degree $k$ with non-negative coefficients ($\lambda_k \neq 0$), Theorem~\ref{Th-2.1} specializes to \cite[Theorem~2.2]{Ahamed-Allu-JRMS-2024}.
\end{itemize}
Consequently, Theorem~\ref{Th-2.1} provides a comprehensive and unified framework extending several existing results concerning the Bohr inequality.
\end{rem}
In view of the sharp bounds of $|f(z)|$ in Lemma B, considering integral powers of $|f(z)|$, we obtain the following result.
\begin{thm}\label{Th-2.2}
	Let $f\in\mathcal{P}_{\mathcal{H}}^{0}(\alpha)$, where $0\leq\alpha< 1$, be given by \eqref{eq-1.3} and
	$\mu:[0,\infty)\to[0,\infty)$ be monotone increasing functions.
	Then, for $p\geq 1$, we have
	\begin{align}\label{eq-2.2}
		\mathcal{I}_{f, p, \mu}(r) := |f(z)|^p + \sum_{n=2}^{\infty} (|a_n| + |b_n|) |z|^n + \mu\left(\frac{S_{|z|}}{\pi}\right) \le d(f(0), \partial f(\mathbb{D})) 
	\end{align}
	\textit{holds for $|z| = r \le R_f^*(\alpha)$, where $R_f^*(\alpha)$ is the unique root in $(0, 1)$ of the equation}
	\begin{align*}
		J_2(r) := &\left(r - 2(1-\alpha)(r+\ln(1-r))\right)^{p} - 2(1-\alpha)(r+\ln(1-r))+ \mu \left(M_{\alpha}(r)\right) \\
		&-1 -2(1-\alpha)(\ln 2 - 1) = 0.
	\end{align*}
	The radius $R_f^*(\alpha)$ is best possible.
\end{thm}
\begin{rem}
In particular, when $\mu(s) \equiv 0$, inequality~\eqref{eq-2.2} of Theorem~\ref{Th-2.2} provides the harmonic analogues of the results established in \cite[Corollary~1]{Kayumov-Khammatova-Ponnusamy-JMAA-2021} for $p = 1$ and $p = 2$. Moreover, for the choice $p = 1$ and $\mu(s) = 2(\sqrt{5} - 1)s$, inequality~\eqref{eq-2.2} reduces to the harmonic counterpart of \cite[Theorem~3, Eq.~(5)]{Ismagilov-Kayumov-Ponnusamy-JMAA-2020}.
\end{rem}
\begin{table}[htbp]
	\centering
	\caption{Numerical values of the sharp radius $R_f^{*}(\alpha)$ for
		$p=1$ and $p=2$ with $\mu(s)\equiv0$.}
	\label{tab:theorem2.2-mu0}
	\begin{tabular}{c|cc}
		\hline
		$\alpha$ & $R_f^{*}(\alpha)$ for $p=1$
		& $R_f^{*}(\alpha)$ for $p=2$ \\
		\hline
		$0.0$ & $0.243755$ & $0.338055$ \\
		$0.2$ & $0.311487$ & $0.399540$ \\
		$0.4$ & $0.386618$ & $0.466807$ \\
		$0.6$ & $0.478477$ & $0.548727$ \\
		$0.9$ & $0.718738$ & $0.761963$ \\
		\hline
	\end{tabular}
\end{table}
\begin{figure}[htbp]
	\centering
	\includegraphics[width=0.98\textwidth]{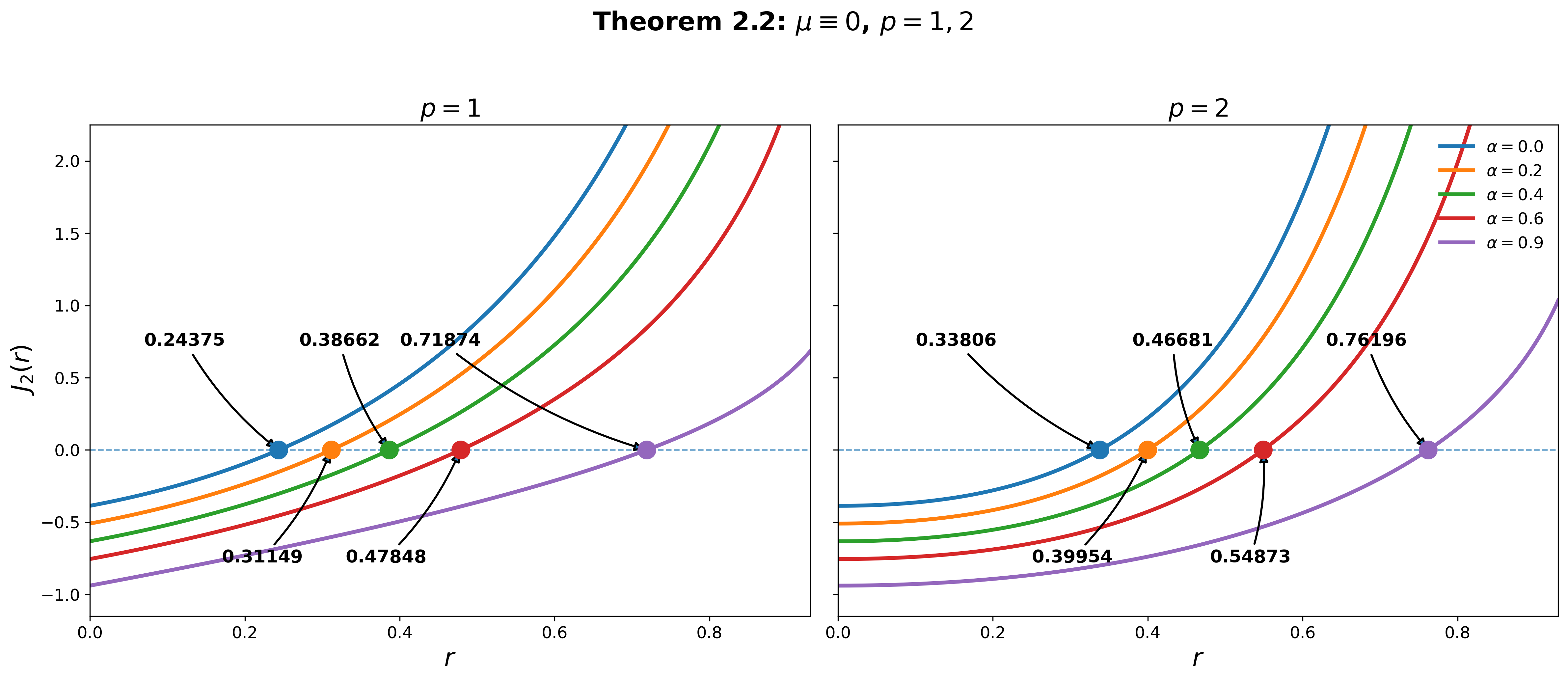}
	\caption{Graphical representation of the equation $J_2(r)=0$ in Theorem~2.2
		for $p=1$ and $p=2$, with $\mu(s)\equiv0$, corresponding to
		$\alpha=0,\,0.2,\,0.4,\,0.6,$ and $0.9$. The marked points indicate
		the corresponding sharp radii $R_f^{*}(\alpha)$.}
	\label{fig:theorem2.2-mu0}
\end{figure}
\subsection{\bf Proof of the main results}
Before, starting the proof, we need some preparation. For $f \in \mathcal{P}_{\mathcal{H}}^{0}(\alpha)$, the Jacobian of $f$ is denoted by $J_f$ which is defined by
\begin{equation*}
	J_f(z) := |f_z(z)|^2 - |f_{\bar{z}}(z)|^2 = |h'(z)|^2 - |g'(z)|^2 \quad \text{for } z \in \mathbb{D}.
\end{equation*}
It is well-known that (see [\cite{Duren-CUP-2004}, p.113]) the area of the image of disk $\mathbb{D}_r := \{z \in \mathbb{C} : |z| < r\}$ under the harmonic map $f = h + \bar{g}$ is
\begin{align}\label{eq-2.3}
	S_r = \iint_{\mathbb{D}_r} J_f(z)\,dxdy = \iint_{\mathbb{D}_r} \left(|h'(z)|^2 - |g'(z)|^2\right) dxdy. 
\end{align}
In view of Lemma A and \eqref{eq-2.3}, we obtain
\begin{align}\label{eq-2.4}
	\frac{S_r}{\pi} = \frac{1}{\pi} \iint_{\mathbb{D}_r} \left(|h'(z)|^2 - |g'(z)|^2\right) dxdy \le r^2+4(1-\alpha)^2\sum_{n=2}^{\infty}\frac{r^{2n}}{n}:=M_{\alpha}(r)
\end{align}
Let $f=h+\overline{g} \in \mathcal{P}^{0}_{\mathcal{H}}(\alpha)$ with $0\leq \alpha <1$ be given by  \eqref{eq-1.3}.
Then in view of Lemma B, we have 
\begin{align}\label{eq-2.5}
	|f(z)|\geq|z|+ \sum\limits_{n=2}^{\infty}  \dfrac{2(1-\alpha)(-1)^{n-1}}{n} |z|^{n},\;\mbox{for}\;|z|<1.
\end{align}
By taking $\liminf$ as $|z|\rightarrow 1$ on the both sides of \eqref{eq-2.5}, we obtain 
\begin{align}\label{eq-2.6}
	\liminf_{|z|\rightarrow1} |f(z)|\geq1+ \sum\limits_{n=2}^{\infty}  \dfrac{2(1-\alpha)(-1)^{n-1}}{n}.
\end{align}
The Euclidean distance between $f(0)$ and the boundary of $f(\mathbb{D})$ is given by 
\begin{align}\label{eq-2.7}
	d(f(0),\partial f(\mathbb{D}))=\liminf_{|z|\rightarrow1}|f(z)-f(0)|.
\end{align}
Since $f(0)=0$, and hence, we have $|f(z)-f(0)|=|f(z)|$. Therefore, from \eqref{eq-2.6} and \eqref{eq-2.7}, we obtain
\begin{align}\label{eq-2.8}
	d(f(0),\partial f(\mathbb{D}))\geq 1+ \sum\limits_{n=2}^{\infty}  \dfrac{2(1-\alpha)(-1)^{n-1}}{n}=1 + 2(1-\alpha)(\ln 2 - 1).
\end{align}
\begin{proof}[\bf Proof of the Theorem \ref{Th-2.1}]
	In view of Lemma A, for $|z| = r$, from \eqref{eq-2.4} we obtain
	\begin{align}\label{eq-2.9}
		 \mathcal{A}_{f, p, \mu, \lambda}(r) &\le r^p + \sum_{n=2}^{\infty} \frac{2(1-\alpha)r^n}{n} + \mu \left( M_\alpha(r) \right) + \lambda \left( (M_\alpha(r))^2 \right) \\&=\nonumber r^p - 2(1-\alpha)(r+\ln(1-r))+ \mu \left( M_\alpha(r) \right) + \lambda \left( (M_\alpha(r))^2 \right):= \mathcal{A}^*_{f, p, \mu, \lambda}(r).
	\end{align}
	In fact, using \eqref{eq-2.8}, it is easy to see that the desired inequality
	\begin{align*}
		\mathcal{A}^*_{f, p, \mu, \lambda}(r) \le 1 + \sum_{n=2}^{\infty} \frac{2(1-\alpha)(-1)^{n-1}}{n}=1 + 2(1-\alpha)(\ln 2 - 1) \le d(f(0), \partial f(\mathbb{D}))
	\end{align*}
	holds for $r \le R_f(\alpha)$, where $R_f(\alpha)$ is the smallest root of the equation $J_1(r) = 0$, where $J_1 : [0, 1] \to \mathbb{R}$ is defined in the statement of the theorem.\vspace{1.2mm}
	
	We now establish the uniqueness of $R_f(\alpha)$ in the interval $(0, 1)$. Note that $J_1(r)$ is a real-valued, differentiable function on $(0, 1)$ with $J_1'(r) > 0$ everywhere in $(0, 1)$. Indeed, direct calculation yields $M_\alpha(r) > 0$ and
	\[
	\frac{d}{dr}\big(M_\alpha(r)\big) = 2r+8(1-\alpha)^2\sum_{n=2}^{\infty}r^{2n-1}> 0, \quad r \in (0, 1).
	\]
	Hence, for $r \in (0, 1)$, given that $\mu$ and $\lambda$ are monotonically increasing functions, with $\mu'(s) > 0$ and $\lambda'(s) > 0$ for all $s \in [0, \infty)$, we have
	\begin{align*}
		\frac{d}{dr}\big(J_1(r)\big) &= pr^{p-1}+\frac{2(1-\alpha)r}{1-r}+ \mu'\big(M_\alpha(r)\big)\frac{d}{dr}\big(M_\alpha(r)\big) \\&\;\;+ 2\lambda'\big((M_\alpha(r))^2\big)M_\alpha(r)\frac{d}{dr}\big(M_\alpha(r)\big) > 0.
	\end{align*}
	Applying the Intermediate Value Theorem, $R_f(\alpha)$ must be the unique solution to $J_1(r) = 0$ in $(0, 1)$. As a result,
	\begin{align}\label{eq-2.10}
			R_f^p(\alpha) - &2(1-\alpha)(R_f(\alpha)+\ln(1-R_f(\alpha)))
		+\mu\big(M_\alpha(R_f(\alpha))\big)
		+\lambda\big(M_\alpha(R_f(\alpha))^2\big)
		\\&\quad\nonumber=1 +2(1-\alpha)(\ln 2 - 1)
	\end{align}
	\vspace{1.2mm}
	
To establish the sharpness of the radius $R_f(\alpha)$, we consider the extremal function $f = f_\alpha$ given in \eqref{eq-1.6}. It is easy to see that $f_\alpha \in \mathcal{P}_{\mathcal{H}}^0(\alpha)$. For $f = f_\alpha$, it can be shown that
\begin{align}\label{eq-2.11}
	d(f(0), \partial f(\mathbb{D})) = 1 + 2(1 - \alpha)(\ln 2 - 1).
\end{align}
Setting $f = f_\alpha$ and taking $|z| = r > R_f(\alpha)$, relations \eqref{eq-2.10} and \eqref{eq-2.11} lead to
\begin{align*}
	\mathcal{A}_{f_\alpha, \mu, \lambda}(R_f(\alpha)) &> R_f^p(\alpha) + \sum_{n=2}^{\infty} (|a_n| + |b_n|) R_f^n(\alpha) + \mu\left(\frac{S_{R_f(\alpha)}}{\pi}\right) + \lambda\left(\left(\frac{S_{R_f(\alpha)}}{\pi}\right)^2\right) \\
	&= R_f^p(\alpha) + \sum_{n=2}^{\infty} \frac{2(1-\alpha)R_f^n(\alpha)}{n} + \mu\big(M_\alpha(R_f(\alpha))\big) + \lambda\big((M_\alpha(R_f(\alpha)))^2\big) \\
	&= 1 + \sum_{n=2}^{\infty} \frac{2(1-\alpha)(-1)^{n-1}}{n} \\&=1 + 2(1 - \alpha)(\ln 2 - 1)\\
	&= d\big(f_\alpha(0), \partial f_\alpha(\mathbb{D})\big),
\end{align*}
which confirms that $R_f(\alpha)$ cannot be improved. This finishes the proof.
\end{proof}
\begin{proof}[\bf Proof of the Theorem \ref{Th-2.2}]
	Applying Lemma A for $|z| = r$, it follows from \eqref{eq-2.4} that
	\begin{align}\label{eq-2.12}
		\mathcal{I}_{f, p, \mu}(r) &\le \left( r + \sum_{n=2}^{\infty} (|a_n| + |b_n|) r^n \right)^p + \sum_{n=2}^{\infty} (|a_n| + |b_n|) r^n + \mu\big(M_\alpha(r)\big)  \\
		&\le \left( r + \sum_{n=2}^{\infty} \frac{2(1-\alpha)r^n}{n} \right)^p + \sum_{n=2}^{\infty} \frac{2(1-\alpha)r^n}{n} + \mu\big(M_\alpha(r)\big) \notag \\
		&\le 1 + \sum_{n=2}^{\infty} \frac{2(1-\alpha)(-1)^{n-1}}{n} \notag\\&=1 + 2(1 - \alpha)(\ln 2 - 1)\notag
	\end{align}
	provided $r \le R_f^*(\alpha)$, where $R_f^*(\alpha)$ represents the smallest root of $J_2(r) = 0$ in the interval $(0, 1)$, and $J_2 : [0, 1] \to \mathbb{R}$ is as defined in the theorem hypothesis. Following the similar argument that being used in Theorem \ref{Th-2.2}, one can readily verify that $R_f^*(\alpha)$ is indeed the unique root of $J_2(r) = 0$ in $(0, 1)$. Hence,
	\begin{align}\label{eq-2.13}
		&\left( R_f^*(\alpha) + \sum_{n=2}^{\infty} \frac{2(1-\alpha)(R_f^*(\alpha))^n}{n} \right)^p + \sum_{n=2}^{\infty} \frac{2(1-\alpha)(R_f^*(\alpha))^n}{n} + \mu\big(M_\alpha(R_f^*(\alpha))\big)  \\
		&= 1 + 2(1 - \alpha)(\ln 2 - 1).\nonumber
	\end{align}
	Thus, the inequality is holds for $r \le R_f^*(\alpha)$.\vspace{1.2mm}
	
	To show the sharpness of the radius $R_f^*(\alpha)$, we test the extremal function $f = f_\alpha$ defined in \eqref{eq-1.6}. From relations \eqref{eq-2.11} and \eqref{eq-2.13}, choosing $f = f_\alpha$ and $|z| = r > R_f^*(\alpha)$ gives
	\begin{align*}
		&\mathcal{I}_{f_\alpha, p, \mu}(R_f^*(\alpha)) \\&> |f_\alpha(R_f^*(\alpha))|^p + \sum_{n=2}^{\infty} (|a_n| + |b_n|) (R_f^*(\alpha))^n + \mu\left( \frac{S_{R_f^*(\alpha)}}{\pi} \right) \\
		&= \left( R_f^*(\alpha) + \sum_{n=2}^{\infty} \frac{2(1-\alpha)(R_f^*(\alpha))^n}{n} \right)^p + \sum_{n=2}^{\infty} \frac{2(1-\alpha)(R_f^*(\alpha))^n}{n} + \mu\big(M_\alpha(R_f^*(\alpha))\big) \\
		&= 1 + \sum_{n=2}^{\infty} \frac{2(1-\alpha)(-1)^{n-1}}{n} \\
		&=1 + 2(1 - \alpha)(\ln 2 - 1)\\&= d\big(f_\alpha(0), \partial f_\alpha(\mathbb{D})\big).
	\end{align*}
	This confirms that $R_f^*(\alpha)$ cannot be improved, completing the proof.
\end{proof}

\noindent{\bf Acknowledgment:} The research of the first author is supported by SERB File No. SUR/2022/002244, Govt. of India and the second author is supported by UGC-JRF, File No. 211610135410, Govt. of India.

\vspace{1.2mm}

\noindent\textbf{Compliance of Ethical Standards:}\\

\noindent\textbf{Conflict of interest.} The authors declare that they have no conflicts of interest regarding the publication of this paper.\vspace{1.5mm}

\noindent\textbf{Funds.} No funds.\vspace{1.5mm}

\noindent\textbf{Data availability statement.}  Data sharing not applicable to this article as no datasets were generated or analysed during the current study.

\end{document}